\documentclass[pdflatex,sn-mathphys-num]{sn-jnl-arxiv}

\usepackage[T1]{fontenc}
\usepackage{amsmath,amssymb,amsthm,mathtools}
\usepackage{booktabs,array,longtable}
\usepackage{enumitem}
\usepackage{xcolor}
\usepackage{hyperref}
\usepackage{graphicx}
\usepackage{CJKutf8} % Retained for compatibility with the v1 bibliography and diff.
\makeatletter
\ifdefined\pdfnobuiltintounicode
  \AddToHook{cmd/CJK@addcmap/after}{\expandafter\pdfnobuiltintounicode\font@name}
\fi
\makeatother

\hypersetup{colorlinks=true,linkcolor=blue!50!black,citecolor=blue!50!black,urlcolor=blue!50!black}
\allowdisplaybreaks[4]

\theoremstyle{thmstyleone}
\newtheorem{theorem}{Theorem}[section]
\newtheorem{maintheorem}{Theorem}

\newtheorem{proposition}[theorem]{Proposition}
\newtheorem{lemma}[theorem]{Lemma}
\newtheorem{corollary}[theorem]{Corollary}
\theoremstyle{thmstylethree}

\theoremstyle{thmstyletwo}
\newtheorem{remark}[theorem]{Remark}

\DeclareMathOperator{\Area}{Area}
\DeclareMathOperator{\Norm}{N}
\newcommand{\Q}{\mathbb{Q}}
\newcommand{\Z}{\mathbb{Z}}

\begin{document}

\title[Rational Area Profiles of Regular Polygons]{Rational Area Profiles of Regular Polygons Cut by Two Congruent Diagonals}

\author{\fnm{Shoichi} \sur{Inoue}}
\email{shoichi31415@gmail.com}
\author{\fnm{Keiichi} \sur{Watanabe}}
\email{keiichi@keiichiw.dev}
% Affiliations and corresponding-author designation have not been supplied.

\abstract{Two congruent diagonals of a regular hexagon cut it into four regions whose areas are all rational multiples of the area of the hexagon. We determine when this phenomenon occurs in a regular $n$-gon. For every regular polygon cut by two intersecting congruent diagonals, neither of which is a diameter, we classify all unions of the four resulting regions whose areas are rational fractions of the total area. The hexagon is the only configuration in which all four regions are simultaneously rational. One of the two regions on the axis of symmetry is rational in exactly five sporadic configurations, with $n=6,18,30$, whereas the other region and the adjacent two-region unions belong to infinite families with simple geometric explanations. The two lateral regions have rational area ratios in exactly two configurations, with $n=6$ and $n=12$. The proof reduces each area ratio to a trigonometric quotient, bounds its rational values by norms in cyclotomic fields, and settles the remaining equations using Myerson's classification of equal products of two sines and the Poonen--Rubinstein classification of a three-sine-product equation.}


% Suggested MSC 2020 codes; confirm before submission.

\maketitle
\section{Introduction}\label{sec:introduction}

Two congruent diagonals of a regular hexagon can cut it into four regions whose areas, as fractions of the whole, are
\[
  \frac1{18},\qquad \frac19,\qquad \frac{13}{18},\qquad \frac19
\]
(Figure~\ref{fig:hexagon}). Is this simultaneous rationality a feature of regular polygons, or an isolated configuration? More generally, which individual regions and which unions of regions have rational area ratios? We classify every such union for the configurations considered below. The hexagon is the unique configuration for which all four individual ratios are rational, but some individual regions and adjacent pairs belong to infinite families.

\begin{figure}[htb]
  \centering
  \includegraphics{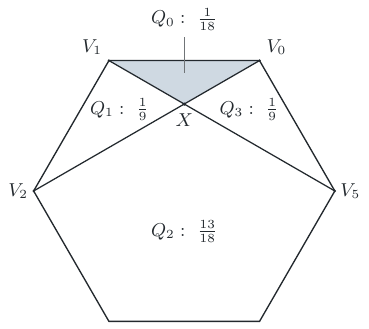}
  \caption{The configuration $(n,k,m)=(6,1,2)$. The indicated areas are normalized by the area of the hexagon.}
  \label{fig:hexagon}
\end{figure}

Let $P_n$ be a regular $n$-gon with centre $O$ and circumradius $1$, with vertices $V_0,\ldots,V_{n-1}$ in counterclockwise order. Throughout, an \emph{admissible configuration} is an integer triple $(n,k,m)$ satisfying
\[
  n\ge5,\qquad 2\le m<\frac n2,\qquad 1\le k<m.
\]
The strict inequality $m<n/2$ excludes diameters.
Draw the congruent diagonals $d_1=V_0V_m$ and $d_2=V_kV_{n-m+k}$, and call their intersection $X$. In cyclic order, let $Q_0,Q_1,Q_2,Q_3$ be the regions adjoining the boundary arcs from $V_0$ to $V_k$, from $V_k$ to $V_m$, from $V_m$ to $V_{n-m+k}$, and from $V_{n-m+k}$ to $V_0$, respectively. With $l=n-2m$ and $t=m-k$, these arcs contain $k,t,k+l,t$ edges; Figure~\ref{fig:configuration} records the notation. Write
\[
  a_i=\frac{\Area(Q_i)}{\Area(P_n)},\qquad
  A_I=\sum_{i\in I}a_i\quad (I\subseteq\{0,1,2,3\}).
\]
Thus $a_0+a_1+a_2+a_3=1$. Reflection interchanges $Q_1$ and $Q_3$, so $a_1=a_3$.

\begin{maintheorem}[Complete rationality profile]\label{thm:profile}
For every admissible configuration, the rational values of $a_0$, $a_2$, $a_1=a_3$, and $a_0+a_1$ occur exactly in the following table. The sets of configurations in its four blocks are denoted by $\mathcal E_0$, $\mathcal E_2$, $\mathcal E_1$, and $\mathcal E_{\mathrm{adj}}$, respectively.
\begin{center}
\small
\begin{tabular}{@{}lll@{}}
\toprule
Area ratio & Configuration $(n,k,m)$ & Value\\
\midrule
$a_0$ \;($\mathcal E_0$) & $(6,1,2)$ & $1/18$\\
 & $(18,3,5)$ & $1/18$\\
 & $(18,4,6)$ & $1/9$\\
 & $(30,3,10)$ & $1/30$\\
 & $(30,5,12)$ & $1/10$\\
\midrule
$a_2$ \;($\mathcal E_2$) & $(6,1,2)$ & $13/18$\\
 & $(9,1,3)$ & $2/3$\\
 & $(n,n/2-3,n/2-2)$, even $n\ge8$ & $1/2+3/n$\\
\midrule
$a_1=a_3$ \;($\mathcal E_1$) & $(6,1,2)$ & $1/9$\\
 & $(12,1,4)$ & $1/6$\\
\midrule
$a_0+a_1$ \;($\mathcal E_{\mathrm{adj}}$) & $(n,k,n/2-1)$, even $n\ge6$ & $1/2-2/n$\\
 & $(12,1,3)$, $(12,2,3)$ & $1/12$\\
\bottomrule
\end{tabular}
\end{center}
The other adjacent sums satisfy
\[
  a_0+a_3=a_0+a_1,\qquad
  a_1+a_2=a_2+a_3=1-(a_0+a_1),
\]
and the opposite sums satisfy
\[
  a_1+a_3=2a_1,\qquad a_0+a_2=1-2a_1.
\]
Consequently, adjacent sums are rational exactly on $\mathcal E_{\mathrm{adj}}$, and opposite sums exactly on $\mathcal E_1$. A three-region sum is rational exactly when its complementary single region is rational. The empty and full sums are always $0$ and $1$. These statements classify $A_I$ for every $I$.
\end{maintheorem}

The main distinction is between the five isolated configurations for $Q_0$ and the infinite families for $Q_2$ and adjacent pairs. Figure~\ref{fig:sporadic-q0} shows all five configurations in $\mathcal E_0$. The latter families have direct geometric explanations: the adjacent union is a half-polygon with one triangle removed, while the $Q_2$ family comes from reflecting a central triangle in a chord (Section~\ref{sec:profile-proof}). This separation into sporadic configurations and one-parameter families also occurs in Poonen--Rubinstein's classification of a three-sine-product equation arising from concurrent polygon diagonals~\cite{PoonenRubinstein1998}. The two classifications address different geometric questions; their connection enters the proof for the lateral regions.

\begin{figure}[!tp]
  \centering
  \includegraphics{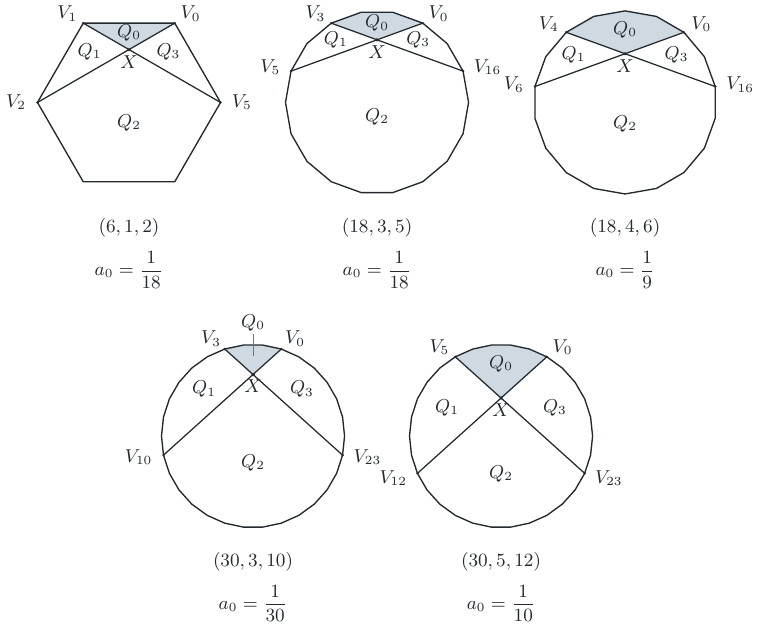}
  \caption{All five configurations with rational $a_0$. In each panel, the reflection axis is vertical and $Q_0$ is shaded. The labels below give $(n,k,m)$ and $a_0=\Area(Q_0)/\Area(P_n)$.}
  \label{fig:sporadic-q0}
\end{figure}

An example involving a regular $18$-gon is discussed by Araki~\cite{Araki2022}. Related work on diagonal lengths provides an arithmetic context for our area problem. Lang~\cite{Lang2012} studies diagonal-to-side length ratios through Chebyshev polynomials and the field $\Q(2\cos(\pi/n))$. Grubb--Wolird~\cite{GrubbWoll2021} study ratios of polygon diagonals using roots of unity, including the degrees of the number fields generated by these ratios. In our simplest case, the adjacent-area problem reduces directly to a quotient of two chord lengths; the other regions require products of such trigonometric quantities.

The proof follows a common pattern. Area decompositions reduce the relevant area-rationality questions to the rationality of certain trigonometric expressions. Cyclotomic norms then restrict the possible rational values to one and a small number of explicitly controlled exceptional cases. The remaining equations are treated using elementary arguments, Myerson's classification of equal products of two sines~\cite{Myerson1993}, and the Poonen--Rubinstein classification of a three-sine-product equation~\cite{PoonenRubinstein1998}. This progression makes explicit the additional arithmetic input required for each geometric quantity.

Section~\ref{sec:geometry} gives the geometric formulas and the notation used throughout. Section~\ref{sec:tools} collects the arithmetic tools. Sections~\ref{sec:adjacent}--\ref{sec:lateral} prove the four component classifications, in the order adjacent sums, $Q_2$, $Q_0$, and lateral regions. Section~\ref{sec:profile-proof} assembles Theorem~\ref{thm:profile} and explains the infinite families geometrically. The final section records further questions, and Appendix~\ref{app:finite-matching} gives the finite comparison used in the classification of the lateral regions.
\section{Regular polygons cut by two congruent diagonals}\label{sec:geometry}

We retain the notation introduced in Section~\ref{sec:introduction}.

\subsection{The four regions}

The admissibility conditions give
\[
  0<k<m<n-m+k<n.
\]
Hence on the circle
\[
  V_0,\quad V_k,\quad V_m,\quad V_{n-m+k}
\]
appear in this order. Since the endpoints of $d_1$ and $d_2$ alternate around the circle, the two chords meet in the interior of the convex polygon $P_n$.
Figure~\ref{fig:configuration} illustrates the construction for the case
$(n,k,m)=(14,3,5)$.

\begin{figure}[htb]
  \centering
  \includegraphics{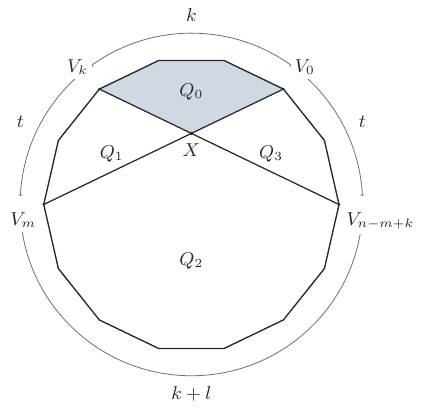}
  \caption{The configuration $(14,3,5)$. The four boundary arcs contain $k,t,k+l,t$ edges, respectively; the diagonals meet at $X$.}
  \label{fig:configuration}
\end{figure}

\subsection{Basic variables}

Set $\theta:=\pi/(2n)$. For the parameters $l,t$ introduced above, admissibility gives
\[
  l\ge1,
  \qquad
  t\ge1,
  \qquad
  l+2t\le n-2,
  \qquad
  n-l\equiv0\pmod 2.
\]
Conversely, from any $(n,l,t)$ satisfying these, the original $(n,k,m)$ is recovered by
\[
  m=\frac{n-l}{2},
  \qquad
  k=m-t=\frac{n-l-2t}{2}.
\]
Also
\[
  2k+l=n-2t.
\]

The following dictionary fixes both the geometric parameters and their arithmetic aliases. Arc lengths here mean numbers of polygon edges.
\begin{center}
\small
\begin{tabular}{@{}lll@{}}
\toprule
Symbol & Definition & Geometric meaning\\
\midrule
$m$ & $2\le m<n/2$ & Minor-arc length of either diagonal\\
$k$ & $1\le k<m$ & Boundary-arc length of $Q_0$\\
$l$ & $n-2m$ & Excess of the $Q_2$ arc over the $Q_0$ arc\\
$t$ & $m-k$ & Boundary-arc length of either lateral region\\
$A$ & $2k$ & Twice the $Q_0$ arc length\\
$B$ & $l$ & Alias for the arc-length excess\\
$C$ & $A+B=n-2t$ & Sum of the $Q_0$ and $Q_2$ arc lengths\\
$\theta$ & $\pi/(2n)$ & One quarter of the central angle\\
\bottomrule
\end{tabular}
\end{center}

\subsection{Symmetry}

\begin{proposition}[Symmetry of the lateral regions]\label{prop:a1a3}
In the above setup,
\[
  a_1=a_3.
\]
\end{proposition}

\begin{proof}
Place the regular $n$-gon on the unit circle and let the argument of $V_j$ be $2\pi j/n$. The reflection across the line of angle $k\pi/n$ sends the point of argument $\varphi$ to the point of argument $2k\pi/n-\varphi$. Hence
\[
  V_0\leftrightarrow V_k,
  \qquad
  V_m\leftrightarrow V_{n-m+k}.
\]
This reflection interchanges the two diagonals, fixes the intersection point $X$, and interchanges the regions $Q_1$ and $Q_3$. Since reflection preserves area, $a_1=a_3$.
\end{proof}

\subsection{Region sums that require classification}

By Proposition~\ref{prop:a1a3} and complementation, it suffices to classify the rationality of
\[
  a_0,\qquad a_2,\qquad a_1,\qquad a_0+a_1.
\]
All remaining region sums are recovered using the identities in Theorem~\ref{thm:profile}.

\subsection{Area formulas and reduction}

The total area of the regular $n$-gon is
\[
  \Area(P_n)=\frac n2\sin\frac{2\pi}{n}=\frac n2\sin(4\theta).
\]
In what follows, when a boundary arc contains $s$ edges, we write the signed area of the polygonal cap bounded by that arc and the corresponding chord as
\[
  \frac{s}{2}\sin\frac{2\pi}{n}-\frac12\sin\frac{2s\pi}{n}.
\]
We use this expression over the full range $0<s<n$. In particular, when $s>n/2$ the quantity $\sin(2s\pi/n)$ may be negative, but this is the signed expression that incorporates the orientation of the central triangle $\triangle OV_iV_{i+s}$, and it gives the actual cap area correctly. Moreover, since the four points $V_0,V_k,V_m,V_{n-m+k}$ lie on the circle in this order, the intersection point $X$ lies on the side of the chords $V_0V_k$ and $V_mV_{n-m+k}$ opposite to the corresponding boundary arc. Therefore the decomposition ``sum of a cap and a triangle'' used in each formula below applies directly without any change of sign.

With the following formulas we reduce all the rationality problems to the rationality of trigonometric quantities.

\begin{proposition}[Area reduction]\label{prop:area-reduction}
\label{prop:a0-formula}\label{prop:a2-formula}\label{prop:a1-formula}
\label{prop:adj-formula}
With the notation above, the four basic area ratios are as follows.
\begin{enumerate}[label=(\roman*)]
\item $a_0=k/n-2R/n$, where
\[
 R=R(n,k,l)=\frac{\sin(2k\theta)\sin(l\theta)}{\sin(4\theta)\sin((2k+l)\theta)}.
\]
\item $a_2=(k+l)/n+2S/n$, where
\[
 S=S(n,l,t)=\frac{\sin(l\theta)\cos((2t-l)\theta)}{\sin(4\theta)\cos(2t\theta)}.
\]
\item $a_1=a_3=t/n-T/n$, where
\[
 T=T(n,l,t)=\frac{2\sin^2(l\theta)\tan(2t\theta)}{\sin(4\theta)}.
\]
\item $a_0+a_1=m/n-U/n$, where
\[
 U=U(n,m)=\frac{\sin(2m\pi/n)}{\sin(2\pi/n)}.
\]
\end{enumerate}
In each item, the area ratio is rational if and only if the displayed trigonometric quantity is rational.
\end{proposition}

\begin{proof}
\noindent\textit{(i)} $Q_0$ is the cap bounded by the arc $V_0\to V_k$ and the chord $V_0V_k$, together with the triangle $\triangle V_0X V_k$. Since the arc length is $k$, the area of the cap is
\[
  \frac{k}{2}\sin\frac{2\pi}{n}-\frac12\sin\frac{2k\pi}{n}.
\]
Setting $\alpha=(m-k)\pi/n$ and $\beta=k\pi/n$, the length of the base $V_0V_k$ is $2\sin\beta$, and by the inscribed angle theorem the base angles of the triangle $\triangle V_0XV_k$ are both $\alpha$. Hence the area of the triangle is $\sin^2\beta\tan\alpha$. Adding these and substituting
\[
  \alpha+\beta=\frac{m\pi}{n}=\frac\pi2-\frac{l\pi}{2n},
  \qquad
  \alpha=\frac\pi2-\frac{(2k+l)\pi}{2n},
\]
we obtain the asserted formula.

\noindent\textit{(ii)} The length of the boundary arc of $Q_2$ is
\[
  (n-m+k)-m=k+l.
\]
Hence the area of the cap bounded by the arc $V_m\to V_{n-m+k}$ and the chord $V_mV_{n-m+k}$ is
\[
  \frac{k+l}{2}\sin\frac{2\pi}{n}-\frac12\sin\frac{2(k+l)\pi}{n}.
\]
The base of the triangle $\triangle V_mXV_{n-m+k}$ is $2\sin((k+l)\pi/n)$ and its two base angles are $t\pi/n$, so its area is
\[
  \sin^2\frac{(k+l)\pi}{n}\tan\frac{t\pi}{n}.
\]
Therefore
\[
 a_2=\frac{k+l}{n}+
 \frac{-\sin(2(k+l)\pi/n)+2\sin^2((k+l)\pi/n)\tan(t\pi/n)}{n\sin(2\pi/n)}.
\]
Here, using
\[
  \frac{(k+l)\pi}{n}=\frac\pi2+(l-2t)\theta,
  \qquad
  \frac{t\pi}{n}=2t\theta
\]
and simplifying, we obtain the stated formula.

\noindent\textit{(iii)} $Q_1$ is the cap bounded by the arc $V_k\to V_m$ and the chord $V_kV_m$, together with the triangle $\triangle V_kXV_m$. Set $\tau=t\pi/n=2t\theta$ and $\lambda=l\pi/(2n)=l\theta$. The area of the cap is
\[
  \frac t2\sin\frac{2\pi}{n}-\frac12\sin(2\tau).
\]
Also, the base of the triangle $\triangle V_kXV_m$ is $2\sin\tau$ and its base angles are
\[
  \frac{k\pi}{n}=\frac\pi2-\tau-\lambda,
  \qquad
  \frac{(k+l)\pi}{n}=\frac\pi2-\tau+\lambda.
\]
By the law of sines, or by elementary trigonometry, the area of the triangle is
\[
  \tan\tau\, (\cos^2\tau-\sin^2\lambda).
\]
Therefore
\[
 \Area(Q_1)=\frac t2\sin\frac{2\pi}{n}-\frac12\sin(2\tau)
 +\tan\tau(\cos^2\tau-\sin^2\lambda).
\]
Since $\tan\tau\cos^2\tau=\frac12\sin(2\tau)$, the $\sin(2\tau)$ terms cancel, and
\[
  \Area(Q_1)=\frac t2\sin\frac{2\pi}{n}-\sin^2\lambda\tan\tau.
\]
Dividing by $\Area(P_n)$ gives the assertion.

\noindent\textit{(iv)} $Q_0\cup Q_1$ does not depend on the second diagonal $d_2$; it is the region bounded by the diagonal $V_0V_m$ and the boundary polyline $V_0V_1\cdots V_m$. The area of this region is, subtracting the triangle $\triangle OV_0V_m$ from the sum of the $m$ central triangles,
\[
  \frac m2\sin\frac{2\pi}{n}-\frac12\sin\frac{2m\pi}{n}.
\]
Dividing by $\Area(P_n)$ gives the assertion.

\end{proof}

\section{Arithmetic tools}\label{sec:tools}

\subsection{Cyclotomic field norms}

Let $N\ge2$, $\zeta=\exp(2\pi i/N)$, and $K=\Q(\zeta)$. For $0<r<N$, write
\[
  d_N(r)=\frac{N}{\gcd(N,r)},\qquad
  \delta_p(r)=
  \begin{cases}
    1/\varphi(d_N(r)),&d_N(r)\text{ is a power of }p,\\
    0,&\text{otherwise}.
  \end{cases}
\]
The dependence of $\delta_p$ on $N$ will always be clear from the choice of $\zeta$.

\begin{lemma}[Cyclotomic norms]\label{lem:cyclotomic-norm}
For $0<r<N$,
\[
 \left|\Norm_{K/\Q}(1-\zeta^r)\right|
 =\begin{cases}
 p^{\varphi(N)/\varphi(p^a)},&d_N(r)=p^a,\\
 1,&d_N(r)\text{ is not a prime power}.
 \end{cases}
\]
\end{lemma}

\begin{proof}
Put $d=d_N(r)$ and $L=\Q(\zeta^r)$. Since $\zeta^r$ is a primitive $d$th root of unity, its conjugates give
\[
 \Norm_{L/\Q}(1-\zeta^r)
 =\prod_{\substack{1\le a\le d\\\gcd(a,d)=1}}(1-\zeta^{ra})
 =\Phi_d(1).
\]
The standard evaluation of $\Phi_d(1)$ is $p$ when $d$ is a power of the prime $p$, and $1$ otherwise; see Bzd\k{e}ga--Herrera-Poyatos--Moree~\cite{BzdegaHerreraPoyatosMoree2018}. By the tower formula for norms and $[K:L]=\varphi(N)/\varphi(d)$,
\[
 \left|\Norm_{K/\Q}(1-\zeta^r)\right|
 =\Phi_d(1)^{[K:L]}
 =\Phi_d(1)^{\varphi(N)/\varphi(d)},
\]
which proves the assertion.
\end{proof}

\begin{lemma}[Restriction of rational quotients]\label{lem:norm-restriction}
Suppose that
\[
 q=\xi\frac{\prod_{i=1}^{h}(1-\zeta^{r_i})}
                  {\prod_{j=1}^{s}(1-\zeta^{u_j})}\in\Q_{>0},
\]
where $\xi\in K$ is a root of unity and all exponents lie strictly between $0$ and $N$. Factors are counted with multiplicity. Then, for every prime $p$,
\[
 v_p(q)=\sum_{i=1}^{h}\delta_p(r_i)-\sum_{j=1}^{s}\delta_p(u_j).
 \tag{\ref{lem:norm-restriction}.1}
\]
In particular:
\begin{enumerate}[label=(\roman*)]
\item If $h,s\le3$, then $v_p(q)=0$ for every $p\ge5$.
\item Each nonzero $\delta_3$ is at most $1/2$, with equality exactly at order $3$. Each nonzero $\delta_2$ is at most $1/2$ when order $2$ is excluded, with equality exactly at order $4$.
\item For $p=3$, or for $p=2$ when all factors have order different from $2$, if $h\le2$ then $v_p(q)=1$ forces $h=2$, both numerator contributions to equal $1/2$, and all denominator contributions to vanish. Analogously, if $s\le2$ then $v_p(q)=-1$ forces $s=2$, both denominator contributions to equal $1/2$, and all numerator contributions to vanish.
\end{enumerate}
\end{lemma}

\begin{proof}
Take absolute norms and divide their prime valuations by $[K:\Q]=\varphi(N)$. Since $\delta_p\le1/(p-1)$, the right-hand side of (\ref{lem:norm-restriction}.1) lies in $[-s/(p-1),h/(p-1)]$. Integrality proves (i); the possible orders give (ii), and equality in these bounds gives (iii).
\end{proof}

The quotients $R$ and $S$ below have two numerator factors and two denominator factors. The quotient $T$ has three numerator factors, counted with multiplicity, so its analysis at $2$ and $3$ must retain this multiplicity.

\subsection{Myerson's two-sine-product classification}
\label{subsec:myerson}

We use Myerson's classification of equal products of two sines at
rational angles~\cite[Theorem~4]{Myerson1993}. We divide the angles by
$\pi$ and state only the positive-angle case. The corrected table and an
explicit description of the allowed permutations appear in
Luk--Cheung~\cite[Theorem~2 and Table~5]{LukCheung2024}.

\begin{theorem}[Myerson]\label{thm:Myerson}
Let $x_1,x_2,x_3,x_4\in\Q\cap(0,1/2]$. The equality
\[
 \sin(\pi x_1)\sin(\pi x_2)
 =\sin(\pi x_3)\sin(\pi x_4)
\]
holds if and only if, up to interchanging the two entries within either
pair and interchanging the two pairs, $(x_1,x_2;x_3,x_4)$ has one of
the following forms, with parameters $a,b,\sigma\in\Q\cap(0,1/2]$:
\begin{enumerate}[label=\textup{(\roman*)}]
\item $(a,b;a,b)$;
\item $\left(\frac16,\sigma;\frac{\sigma}{2},\frac{1-\sigma}{2}\right)$;
\item one of the fifteen rows in Table~\ref{tab:myerson}.
\end{enumerate}
\end{theorem}

\begin{table}[htbp]
\centering
\caption{The sporadic solutions in Myerson's classification.
The entries are angles divided by $\pi$. The numbering follows
Luk--Cheung~\cite[Table~5]{LukCheung2024}.}
\label{tab:myerson}
\begin{tabular}{@{}rcccc@{}}
\toprule
$j$ & $x_1$ & $x_2$ & $x_3$ & $x_4$\\
\midrule
1  & $1/21$ & $8/21$  & $1/14$ & $3/14$\\
2  & $1/14$ & $5/14$  & $2/21$ & $5/21$\\
3  & $4/21$ & $10/21$ & $3/14$ & $5/14$\\
4  & $1/20$ & $9/20$  & $1/15$ & $4/15$\\
5  & $2/15$ & $7/15$  & $3/20$ & $7/20$\\
6  & $1/30$ & $3/10$  & $1/15$ & $2/15$\\
7  & $1/15$ & $7/15$  & $1/10$ & $7/30$\\
8  & $1/10$ & $13/30$ & $2/15$ & $4/15$\\
9  & $4/15$ & $7/15$  & $3/10$ & $11/30$\\
10 & $1/30$ & $11/30$ & $1/10$ & $1/10$\\
11 & $7/30$ & $13/30$ & $3/10$ & $3/10$\\
12 & $1/15$ & $4/15$  & $1/10$ & $1/6$\\
13 & $2/15$ & $7/15$  & $1/6$  & $3/10$\\
14 & $1/12$ & $5/12$  & $1/10$ & $3/10$\\
15 & $1/10$ & $3/10$  & $1/6$  & $1/6$\\
\bottomrule
\end{tabular}
\end{table}

In row~13 the second entry is $7/15$, not $8/15$; this is the
correction noted in~\cite[Section~4, following Theorem~2]{LukCheung2024}.
These pair-preserving permutations give at most eight distinct orderings.
In the applications below, we test
these orderings against an additional linear relation among the angles.
The verification script \texttt{verify\_myerson.py} \cite{Watanabe2026Code} checks the table
identities and the finite comparisons in the proofs below using exact arithmetic.

\subsection{The Poonen--Rubinstein three-sine-product equation}

Poonen and Rubinstein studied concurrent diagonals in regular polygons by reducing concurrency to a three-sine-product equation and then to a vanishing sum of roots of unity of weight $12$. The same equation occurs here in the rationality of the lateral regions.

\begin{theorem}[Poonen--Rubinstein]\label{thm:PR}
The solutions in positive rationals $u_1,u_2,u_3,v_1,v_2,v_3$ of
\[
  u_1+u_2+u_3+v_1+v_2+v_3=1
\]
and
\[
  \sin(\pi u_1)\sin(\pi u_2)\sin(\pi u_3)
  =
  \sin(\pi v_1)\sin(\pi v_2)\sin(\pi v_3)
\]
consist, up to permutation within each of the two triples and interchange of the two triples, of the trivial solutions, four one-parameter families, and $65$ sporadic solutions, as classified in Poonen--Rubinstein~\cite[Theorem~4]{PoonenRubinstein1998}. Here a solution is \emph{trivial} if $(v_1,v_2,v_3)$ is a permutation of $(u_1,u_2,u_3)$, with multiplicities retained; the sum condition then gives $u_1+u_2+u_3=v_1+v_2+v_3=1/2$. The four one-parameter families are listed in Table~3 of that paper, and the sporadic solutions in Table~4.
\end{theorem}

\section{Adjacent two-region sums}\label{sec:adjacent}

\begin{theorem}[Classification of adjacent two-region sums]\label{thm:Adj}
In the above setup, $a_0+a_1\in\Q$ holds if and only if either
\[
  n\text{ is even},
  \qquad
  m=\frac n2-1,
  \qquad
  1\le k<m,
\]
or
\[
  (n,k,m)=(12,1,3),\qquad (12,2,3).
\]
In the former case the value is
\[
  a_0+a_1=\frac12-\frac2n,
\]
and in the latter case the value is
\[
  a_0+a_1=\frac1{12}.
\]
The same classification holds for all adjacent two-region sums
\[
  a_0+a_1,
  \quad a_1+a_2,
  \quad a_2+a_3,
  \quad a_3+a_0.
\]
\end{theorem}

The adjacent sums are governed by the single sine quotient
\[
 U(n,m)=\frac{\sin(2m\pi/n)}{\sin(2\pi/n)}.
\]
If $r_j=\sin(j\pi/n)/\sin(\pi/n)$ denotes the chord length $|V_0V_j|$ divided by the side length, then $U=r_{2m}/r_2$. Thus this first classification connects directly with the arithmetic of diagonal-length ratios studied, for example, in~\cite{GrubbWoll2021}.

\begin{lemma}\label{lem:U-rational}
For $2\le m<n/2$, the quotient $U(n,m)$ is rational exactly when
\[
 n\text{ is even},\quad m=n/2-1,\quad U=1,
\]
or
\[
 (n,m)=(12,3),\quad U=2.
\]
\end{lemma}

\begin{proof}
Setting $\zeta=\exp(2\pi i/n)$,
\[
  U(n,m)=\frac{\zeta^m-\zeta^{-m}}{\zeta-\zeta^{-1}}
  =\zeta^{m-1}+\zeta^{m-3}+\cdots+\zeta^{-(m-1)}.
\]
Hence $U(n,m)$ is an algebraic integer. If $U(n,m)\in\Q$, then $U(n,m)\in\Z$. Also, since $0<2m\pi/n<\pi$, we have $U(n,m)>0$.

We restrict the value using norms. Let $q=U(n,m)\in\Z_{>0}$.
\[
  q=\zeta^{1-m}\frac{1-\zeta^{2m}}{1-\zeta^2}.
\]
Lemma~\ref{lem:norm-restriction} gives, for each prime $p$,
\[
  v_p(q)=\delta_p(2m)-\delta_p(2).
\]
For $p\ge3$, by integrality $v_p(q)=0$. For $p=2$, only $q=1$ or $q=2$ remains.

If $q=1$, then
\[
  \sin\frac{2m\pi}{n}=\sin\frac{2\pi}{n}.
\]
Since $m\ge2$ and $2\pi/n<2m\pi/n<\pi$,
\[
  \frac{2m\pi}{n}=\pi-\frac{2\pi}{n},
\]
i.e. $2m+2=n$. Hence $n$ is even and $m=n/2-1$.

If $q=2$, then the norm condition requires $\zeta^{2m}=-1$, i.e. $2m=n/2$. In this case
\[
  U(n,m)=\frac1{\sin(2\pi/n)}.
\]
For this to be $2$, we need $\sin(2\pi/n)=1/2$, and since $n\ge5$, $n=12$. Hence $m=3$.

Both displayed cases satisfy the asserted identities by direct substitution.
\end{proof}

\begin{proof}[Proof of Theorem~\ref{thm:Adj}]
By Proposition~\ref{prop:adj-formula}, the rationality of $a_0+a_1$ is equivalent to that of $U(n,m)$. By Lemma~\ref{lem:U-rational}, $U$ is rational only when $n$ is even and $m=n/2-1$, or $(n,m)=(12,3)$. The values are obtained by substituting into Proposition~\ref{prop:adj-formula}. For $(12,3)$, from $1\le k<m$ we have $k=1,2$. The adjacent sum $a_2+a_3$ is the total minus $a_0+a_1$, so it has the same rationality. As for $a_1+a_2$ and $a_3+a_0$, we look at the two sides of the other congruent diagonal $d_2$. The endpoints of $d_2$ divide the circle into arcs of length $n-m$ and $m$. The normalized area of the cap-type region on the side of arc length $s$ is
\[
  \frac{s}{n}-\frac1n\frac{\sin(2s\pi/n)}{\sin(2\pi/n)}.
\]
Setting $s=n-m$,
\[
  \sin\frac{2(n-m)\pi}{n}=-\sin\frac{2m\pi}{n},
\]
so the rationality reduces to the rationality of the same sine quotient $U(n,m)$. Hence we obtain the same classification for all adjacent two-region sums.
\end{proof}

\section{The region \texorpdfstring{$Q_2$}{Q2}}\label{sec:q2}

\begin{theorem}[Classification of $Q_2$]\label{thm:Q2}
In the above setup, $a_2\in\Q$ holds if and only if one of the following cases occurs:
\[
  (n,k,m)=(6,1,2),\qquad (9,1,3),
\]
or, for even $n\ge8$,
\[
  k=\frac n2-3,
  \qquad
  m=\frac n2-2.
\]
The corresponding values are, respectively,
\[
  \frac{13}{18},\qquad \frac23,
  \qquad \frac12+\frac3n.
\]
\end{theorem}

By Proposition~\ref{prop:a2-formula}, it suffices to classify the rational values of $S(n,l,t)$. We use the same two stages as for the adjacent sine quotient: a norm restriction followed by the classification of the surviving value.

\subsection{Restriction of the rational value}

\begin{lemma}\label{lem:S-rational-values}
Under the conditions
\[
  n\ge5,
  \quad l\ge1,
  \quad t\ge1,
  \quad l+2t\le n-2,
  \quad n-l\equiv0\pmod2,
\]
suppose $S(n,l,t)\in\Q$. Then
\[
  S(n,l,t)=1
\]
or
\[
  (n,l,t)=(6,2,1),\qquad S(n,l,t)=\frac23.
\]
\end{lemma}

\begin{proof}
Set $q=S(n,l,t)\in\Q$. Since $S>0$, we have $q>0$. Set $N=2n$ and $\zeta=\exp(2\pi i/N)$. Then
\[
  S=\zeta^2
  \frac{(1-\zeta^l)(1+\zeta^{2t-l})}{(1-\zeta^4)(1+\zeta^{2t})}.
\]
Since $\zeta^n=-1$,
\[
  1+\zeta^{2t-l}=1-\zeta^{n+2t-l},
  \qquad
  1+\zeta^{2t}=1-\zeta^{n+2t}.
\]
Thus, for each prime $p$, Lemma~\ref{lem:norm-restriction} gives
\[
  v_p(q)=\delta_p(l)+\delta_p(n+2t-l)-\delta_p(4)-\delta_p(n+2t).
  \tag{\ref{lem:S-rational-values}.1}
\]

Lemma~\ref{lem:norm-restriction}(i) gives $v_p(q)=0$ for $p\ge5$.

Next consider $p=2$. For $l,4,n+2t$, the case $d_N(r)=2$ does not occur. On the other hand, $d_N(n+2t-l)=2$ occurs only when
\[
  n+2t-l=n,
  \qquad\text{i.e.}\qquad l=2t.
\]

First we show that $v_2(q)<0$ cannot occur. If it did, since the right-hand side is an integer, $v_2(q)=-1$. As each negative term is at most $1/2$, this requires
\[
  \delta_2(4)=\delta_2(n+2t)=\frac12.
\]
In particular $d_N(4)=4$, and under $n\ge5$ this forces $N=16$, i.e. $n=8$. Furthermore, from $d_{16}(8+2t)=4$ we get $t=2$. By $l+2t\le n-2$ and $n-l\equiv0\pmod2$ we get $l=2$. But in this case
\[
  \delta_2(l)+\delta_2(n+2t-l)-\delta_2(4)-\delta_2(n+2t)
  =\frac14+\frac14-\frac12-\frac12
  =-\frac12,
\]
not an integer, a contradiction.

The positive contributions are at most $1/2+1=3/2$, so integrality gives $v_2(q)\le1$. Next let $v_2(q)=1$. If $l\ne2t$, then the positive terms are also at most $1/2$. To reach the integer $1$ we need
\[
  \delta_2(l)=\delta_2(n+2t-l)=\frac12,
  \qquad
  \delta_2(4)=\delta_2(n+2t)=0.
\]
The former means $l=n/2$ and that $n+2t-l$ is $n/2$ or $3n/2$. These give $t=0$ or $2t=n$ respectively, both contradicting the conditions. Hence $l=2t$. That is,
\[
  v_2(q)=1\quad\Longrightarrow\quad l=2t.
  \tag{\ref{lem:S-rational-values}.2}
\]

For $p=3$, $v_3(q)=1$ is impossible. If both positive contributions were $1/2$, then $d_N(l)=d_N(n+2t-l)=3$; but from $0<l<n$ we get $l=2n/3$, and further $n+2t-l=2n/3$ or $4n/3$, both contradicting $l+2t\le n-2$. Also, if $v_3(q)=-1$ occurred, the two negative terms would both have to be $1/2$. In particular $d_N(4)=3$, so under $n\ge5$ this forces $N=12$, i.e. $n=6$. In this case the conditions give $(l,t)=(2,1)$, and direct computation gives
\[
  S(6,2,1)=\frac23.
\]

By the above, except for the exception $(6,2,1)$, the only prime factor of $q$ is $2$, so $q=1$ or $q=2$. We show that $q=2$ is impossible. If $q=2$, then by (\ref{lem:S-rational-values}.2) we have $l=2t$. In this case
\[
  S=\frac{\tan(2t\theta)}{\sin(4\theta)}.
\]
Hence $S=2$ means
\[
  \tan\left(\frac{t\pi}{n}\right)=2\sin\left(\frac{2\pi}{n}\right).
\]
With $x=\pi/n$ and $r=t$, the admissibility condition gives $1\le r<n/4$. Lemma~\ref{lem:tan-obstruction} below excludes this equality. Therefore $q=2$ is impossible, and apart from the exception $q=1$.
\end{proof}

\begin{lemma}[An elementary inequality]\label{lem:tan-obstruction}
For integers $n\ge5$ and $1\le r<n/4$, with $x=\pi/n$,
\[
 \tan(rx)\ne2\sin(2x).
\]
\end{lemma}

\begin{proof}
For $r=1$, the inequality $\tan x<2\sin2x$ follows from $\cos x>1/2$. For $r=2$, $x<\pi/8$ gives $\tan2x<2\sin2x$. For $r=3$, $x<\pi/12$ and
\[
 2\sin2x-\tan3x=\frac{\sin x(2\cos4x-1)}{\cos3x}>0.
\]
For $r\ge4$, $4x\le rx<\pi/4$, and
\[
 \tan(rx)\ge\tan4x>4x>2\sin2x.
\]
\end{proof}

\subsection{Classification of \texorpdfstring{$S=1$}{S=1}}

\begin{lemma}\label{lem:S-one}
Under the admissibility conditions of
Lemma~\ref{lem:S-rational-values}, $S(n,l,t)=1$ holds if and only if
\[
 (n,l,t)=(9,3,2),
\]
or
\[
 l=4,\qquad t=1,\qquad n\equiv0\pmod2,\qquad n\ge8.
\]
\end{lemma}

\begin{proof}
Put
\[
 a=\frac{l}{2n},\qquad c=\frac2n,\qquad
 d=\frac12-\frac tn,\qquad
 b=\frac12-\left|\frac12-a-d\right|.
\]
Admissibility gives $d-a=(n-l-2t)/(2n)\ge1/n>0$, so
$0<a<d<1/2$, $0<c<1/2$, and $0<b\le1/2$.
Replacing the cosine factors by sines shows that $S=1$ is equivalent to
\[
 \sin(\pi a)\sin(\pi b)=\sin(\pi c)\sin(\pi d).
\]
Moreover, the definition of $b$ is equivalent, in these ranges, to
\begin{equation}\label{eq:S-myerson-test}
 b=a+d\qquad\text{or}\qquad a+b+d=1.
\end{equation}
The original parameters are recovered by
\[
 n=\frac2c,\qquad l=2na,\qquad
 t=n\left(\frac12-d\right).
\]
We apply Theorem~\ref{thm:Myerson}.

\emph{Trivial solutions.}
Since $a<d$, equality of the two unordered pairs forces $a=c$ and
$b=d$. Thus $l=4$ and $|2t-4|=2t$, so $t=1$.
The remaining admissibility conditions give even $n\ge8$.

\emph{The one-parameter family.}
Write $h=(1-\sigma)/2$. The following table lists, for each of the
eight orderings, all roots of \eqref{eq:S-myerson-test} in
$0<\sigma\le1/2$:
\[
\begin{array}{c|c}
 (a,b;c,d)&\sigma\\ \hline
 (1/6,\sigma;\sigma/2,h)&4/9\\
 (1/6,\sigma;h,\sigma/2)&1/3\\
 (\sigma,1/6;\sigma/2,h)&\text{none}\\
 (\sigma,1/6;h,\sigma/2)&1/9\\
 (\sigma/2,h;1/6,\sigma)&1/4,\ 1/2\\
 (\sigma/2,h;\sigma,1/6)&1/3\\
 (h,\sigma/2;1/6,\sigma)&1/2\\
 (h,\sigma/2;\sigma,1/6)&\text{none}
\end{array}
\]
Each row requires solving only two affine equations. The roots
$\sigma=1/3$ give $a=d$, the root $\sigma=1/9$ gives $a>d$, and the
roots $\sigma=1/2$ give $d=1/2$. These contradict $0<a<d<1/2$.
The remaining roots $4/9$ and $1/4$ yield, respectively,
\[
 (n,l,t)=(9,3,2),\qquad (12,3,3).
\]
The second triple is excluded by $n-l\equiv0\pmod2$.

\emph{Sporadic solutions.}
For each row of Table~\ref{tab:myerson}, choose $(a,b)$ as an ordering
of either pair and choose $d$ from the other pair; its remaining entry
is $c$. Testing \eqref{eq:S-myerson-test} in these eight orderings gives
\[
\begin{array}{c|c}
 j&(a,b;c,d)\\ \hline
 1\text{--}11,\ 14,\ 15&\text{none}\\
 12&(1/10,1/6;4/15,1/15)\\
 13&(1/6,3/10;7/15,2/15)
\end{array}
\]
Both assignments have $a>d$, contrary to admissibility.

Conversely, $l=4,t=1$ gives a trivial identity, and $(n,l,t)=(9,3,2)$
gives the one-parameter identity with $\sigma=4/9$. Thus each listed
configuration satisfies $S=1$.
\end{proof}

\begin{proof}[Proof of Theorem~\ref{thm:Q2}]
By Proposition~\ref{prop:a2-formula}, $a_2\in\Q$ and $S(n,l,t)\in\Q$ are equivalent. By Lemma~\ref{lem:S-rational-values}, this is limited to the exception $(n,l,t)=(6,2,1)$ with $S=2/3$, or the case $S=1$. The exception gives
\[
  m=\frac{6-2}{2}=2,
  \qquad k=m-1=1,
\]
with value
\[
  a_2=\frac{k+l}{n}+\frac2nS=\frac36+\frac26\cdot\frac23=\frac{13}{18}.
\]

For the case $S=1$ we use Lemma~\ref{lem:S-one}. From $(n,l,t)=(9,3,2)$ we get $(n,k,m)=(9,1,3)$, with value
\[
  a_2=\frac49+\frac29=\frac23.
\]
Also, in the case $l=4,t=1$ with $n$ even, $n\ge8$,
\[
  m=\frac{n-4}{2}=\frac n2-2,
  \qquad
  k=m-1=\frac n2-3,
\]
and
\[
  a_2=\frac{k+l}{n}+\frac2n
  =\frac{\frac n2-3+4}{n}+\frac2n
  =\frac12+\frac3n.
\]
The reverse direction follows by direct substitution in each case.
\end{proof}

\section{The region \texorpdfstring{$Q_0$}{Q0}}\label{sec:q0}

\begin{theorem}[Classification of $Q_0$]\label{thm:Q0}
In the above setup, $a_0\in\Q$ holds if and only if
\[
  (n,k,m)\in
  \{(6,1,2),(18,3,5),(18,4,6),(30,3,10),(30,5,12)\}.
\]
The corresponding values of $a_0$ are
\[
\begin{array}{c|ccccc}
(n,k,m)&(6,1,2)&(18,3,5)&(18,4,6)&(30,3,10)&(30,5,12)\\
\midrule
 a_0&1/18&1/18&1/9&1/30&1/10
\end{array}.
\]
\end{theorem}

Throughout this section put
\[
 A=2k,\qquad B=l,\qquad C=A+B.
\]
Thus $A\ge2$ is even, $B\ge1$, $C<n$, and
\[
 R=\frac{\sin(A\theta)\sin(B\theta)}{\sin(4\theta)\sin(C\theta)}.
\]
We first restrict its rational value and then use Myerson's theorem to classify the case $R=1$.

\subsection{Restriction of the rational value}

\begin{lemma}\label{lem:R-rational-values}
If $R(n,k,l)\in\Q$, then $R=1$, unless
\[
 (n,k,l)=(6,1,2),\qquad R=\frac13.
\]
\end{lemma}

\begin{proof}
First assume $R=q\in\Q$. Setting $N=2n$ and $\zeta=\exp(2\pi i/N)$,
\[
  R=\zeta^2\frac{(1-\zeta^A)(1-\zeta^B)}{(1-\zeta^4)(1-\zeta^C)}.
\]
By Lemma~\ref{lem:norm-restriction}, for each prime $p$,
\[
  v_p(q)=\delta_p(A)+\delta_p(B)-\delta_p(4)-\delta_p(C).
\]
Lemma~\ref{lem:norm-restriction}(i) gives $v_p(q)=0$ for $p\ge5$.

Next consider $p=2$. Since $A,B,C,4$ are all less than $N/2=n$, the case $d_N(r)=2$ does not occur, so the nonzero $\delta_2$ is at most $1/2$. If $v_2(q)=1$, then the two positive terms must both equal $1/2$. Hence $d_N(A)=d_N(B)=4$, and from $0<A,B<N/2$ we get $A=B=N/4$. But then $C=A+B=N/2=n$, contradicting $C<n$. If $v_2(q)=-1$, then the two negative terms must both equal $1/2$. In particular $d_N(4)=4$. Under $n\ge5$ this forces $N=16$. Furthermore, from $d_N(C)=4$ and $C<n=8$ we get $C=4$. Hence $A+B=4$; since $A\ge2$ is even and $B\ge1$, we have $A=B=2$. But in this case the right-hand side is
\[
  \delta_2(2)+\delta_2(2)-\delta_2(4)-\delta_2(4)
  =\frac14+\frac14-\frac12-\frac12
  =-\frac12,
\]
which is not an integer. Hence $v_2(q)=0$.

Finally consider $p=3$. The nonzero $\delta_3$ is at most $1/2$. If $v_3(q)=1$, then the two positive terms must both equal $1/2$. Hence $d_N(A)=d_N(B)=3$, and from $0<A,B<N/2$ we get $A=B=N/3$. But then $C=2N/3>N/2=n$, a contradiction. If $v_3(q)=-1$, then the two negative terms must both equal $1/2$. In particular $d_N(4)=3$, and under $n\ge5$ this forces $N=12$. Furthermore, from $d_N(C)=3$ and $C<n=6$ we get $C=4$. Hence $A+B=4$, and by the conditions $A=B=2$. This corresponds to
\[
  (n,k,l)=(6,1,2),
\]
and in this case
\[
  R(6,1,2)=\frac13.
\]
Otherwise $v_p(q)=0$ for all primes $p$. Since $q>0$, we get $q=1$.

\end{proof}

\subsection{Classification of \texorpdfstring{$R=1$}{R=1}}

\begin{lemma}\label{lem:R-one}
The equality $R(n,k,l)=1$ holds if and only if
\[
 (n,k,l)=(18,3,8),(18,4,6),(30,3,10),(30,5,6).
\]
\end{lemma}

\begin{proof}
Put
\[
 a=\frac{A}{2n},\qquad b=\frac{B}{2n},\qquad
 c=\frac2n,\qquad d=\frac{C}{2n}.
\]
Then $0<a,b,c,d<1/2$ and $d=a+b$. Hence $R=1$ is equivalent to
\[
 \sin(\pi a)\sin(\pi b)=\sin(\pi c)\sin(\pi d),
 \qquad d=a+b<\frac12.
\]
We apply Theorem~\ref{thm:Myerson}.

\emph{Trivial solutions.}
These are impossible because $d>a,b$.

\emph{The one-parameter family.}
The pair $\{a,b\}$ cannot be
$\{\sigma/2,(1-\sigma)/2\}$, whose sum is $1/2$.
Thus
\[
 \{a,b\}=\{1/6,\sigma\},\qquad
 \{c,d\}=\{\sigma/2,(1-\sigma)/2\}.
\]
Since $d=a+b>\sigma/2$, we must have
\[
 d=\frac{1-\sigma}{2},\qquad c=\frac{\sigma}{2},\qquad
 \frac16+\sigma=\frac{1-\sigma}{2}.
\]
Therefore $\sigma=2/9$, and
\[
 n=\frac2c=18,\qquad \{A,B\}=\{6,8\}.
\]

\emph{Sporadic solutions.}
For an ordering of a row of Table~\ref{tab:myerson} to satisfy $d=a+b$,
the sum of one pair must equal an entry of the other pair. The complete
comparison is
\[
\begin{array}{c|c|c|c}
 j&\{a,b\}&c&d\\ \hline
 1\text{--}11,\ 14,\ 15&\text{none}&\text{--}&\text{--}\\
 12&\{1/10,1/6\}&1/15&4/15\\
 13&\{1/6,3/10\}&2/15&7/15
\end{array}
\]
Row~12 gives $n=30$ and $\{A,B\}=\{6,10\}$.
Row~13 gives $n=15$ and $\{A,B\}=\{5,9\}$, contradicting
$A=2k\equiv0\pmod2$.

Finally, $A=2k$ and $B=l$ give the four triples in the statement.
Each is admissible. The two triples with $n=18$ realize the
one-parameter identity with $\sigma=2/9$, and those with $n=30$
realize row~12. This also proves the converse.
\end{proof}

\begin{proof}[Proof of Theorem~\ref{thm:Q0}]
By Proposition~\ref{prop:a0-formula}, $a_0$ is rational exactly when $R$ is rational. Lemmas~\ref{lem:R-rational-values} and \ref{lem:R-one} give the exception $(n,k,l)=(6,1,2)$ with $R=1/3$ and the four cases with $R=1$. The conversion $m=(n-l)/2$ gives precisely the five configurations in the statement. Substitution into
\[
 a_0=\frac{k}{n}-\frac2nR
\]
gives their asserted values, and the same lemmas establish the converse.
\end{proof}

\section{The lateral regions}\label{sec:lateral}

\begin{theorem}[Classification of the lateral regions]\label{thm:Q1Opp}
In the above setup, the common area ratio $a_1=a_3$ is rational if and only if
\[
  (n,k,m)=(6,1,2)\quad\text{or}\quad(12,1,4).
\]
Its values are $1/9$ and $1/6$, respectively.
\end{theorem}

Set $B=l$ and $C=n-2t=2k+l$. The conditions are
\[
  1\le B<C<n,
  \qquad
  C-B=2k\ge2,
  \qquad
  n-C=2t\ge2,
  \qquad
  B\equiv C\equiv n\pmod2.
\]
Since $C\theta=\pi/2-2t\theta$, the quotient $T=T(n,l,t)$ has the equivalent expression
\[
  T(n,l,t)=T(n,B,(n-C)/2)
  =\frac{2\sin^2(B\theta)\cos(C\theta)}{\sin(4\theta)\sin(C\theta)}.
\]

The proof again separates the restriction of the rational value from the classification of the value $1$.

\subsection{Restriction of the rational value}

\begin{lemma}\label{lem:T-rational-values}
Under the displayed conditions, if $T(n,B,(n-C)/2)\in\Q$, then $T=1$, unless
\[
 (n,B,C)=(6,2,4),\qquad T=\frac13.
\]
\end{lemma}

\begin{proof}
Let $T=q\in\Q$, and set $N=2n$, $\zeta=\exp(2\pi i/N)$.
\[
  T=\zeta^{2-B}\frac{(1-\zeta^B)^2(1-\zeta^{C+n})}{(1-\zeta^4)(1-\zeta^C)}.
\]
By Lemma~\ref{lem:norm-restriction},
\[
  v_p(q)=2\delta_p(B)+\delta_p(C+n)-\delta_p(4)-\delta_p(C).
  \tag{\ref{thm:Q1Opp}.1}
\]
Here the numerator has three factors, including the repeated factor $(1-\zeta^B)$. Lemma~\ref{lem:norm-restriction}(i) still gives $v_p(q)=0$ for $p\ge5$.

Consider $p=2$. For $B,C,4$, the case $d_N(r)=2$ does not occur, and for $C+n$ the case $d_N(C+n)=2$ does not occur either. Hence the nonzero $\delta_2$ is at most $1/2$. If $v_2(q)<0$, then $v_2(q)=-1$, and the two negative terms must both be $1/2$. In particular $d_N(4)=4$, so $N=16$, and from $d_N(C)=4$ and $C<n=8$ we get $C=4$. From $C-B\ge2$ and the parity condition, $B=2$. In this case
\[
  2\delta_2(B)+\delta_2(C+n)-\delta_2(4)-\delta_2(C)
  =2\cdot\frac14+\frac12-\frac12-\frac12=0,
\]
so $v_2(q)\ne-1$. Hence a negative $2$-valuation does not occur. Also, if $v_2(q)=1$, then $\delta_2(B)=1/2$ is required. Indeed, if $\delta_2(B)<1/2$ then $2\delta_2(B)\le1/2$. If $\delta_2(C+n)\le1/4$, the positive contributions sum to at most $3/4$. Otherwise $\delta_2(C+n)=1/2$, which gives $C=n/2$ and hence also $\delta_2(C)=1/2$, so the right-hand side is at most $1/2$. Neither case reaches $1$. Hence $d_N(B)=4$, and from $0<B<n=N/2$,
\[
  v_2(q)=1\quad\Longrightarrow\quad B=\frac n2.
\]

Next consider $p=3$. If $v_3(q)=-1$, the two negative terms must both be $1/2$. In particular $d_N(4)=3$, so under $n\ge5$ this forces $N=12$, i.e. $n=6$. Furthermore, from $d_N(C)=3$ and $C<n=6$ we get $C=4$, and by the conditions $B=2$. Thus $(n,B,C)=(6,2,4)$ and $t=(n-C)/2=1$, giving
\[
  T=\frac13.
\]
On the other hand, if $v_3(q)=1$, then $\delta_3(B)=1/2$ is required: otherwise $\delta_3(B)\le1/6$ and the positive contributions sum to at most $5/6$. Hence $d_N(B)=3$, and from $0<B<n=N/2$,
\[
  v_3(q)=1\quad\Longrightarrow\quad B=\frac{2n}{3}.
\]
For $p=2,3$, the three positive contributions sum to at most $3/2$ and the two negative contributions to at most $1$, so integrality gives $-1\le v_p(q)\le1$. Therefore, excluding the exception $(n,B,C)=(6,2,4)$, the candidates for the rational value are limited to
\[
  q=1,
  \quad q=2,
  \quad q=3,
  \quad q=6.
\]
However, $q=6$ is the case where $v_2(q)>0$ and $v_3(q)>0$ occur simultaneously. In this case the former requires $B=n/2$ and the latter $B=2n/3$ at the same time, which are incompatible. Hence in fact only $q=1,2,3$ need be examined.

If $q=2$, then by the above necessary condition for the $2$-valuation, $B=n/2$. In this case $2\sin^2(B\theta)=1$, and $T=2$ gives
\[
  \cot(C\theta)=2\sin(4\theta).
\]
Set $r=(n-C)/2$ and $x=\pi/n$. Since $n-C\ge2$ and $C>B=n/2$, we have $1\le r<n/4$, and
\[
  \tan(rx)=2\sin(2x).
\]
Lemma~\ref{lem:tan-obstruction} excludes this equation. The case $q=3$ likewise gives $B=2n/3$ by the necessary condition for the $3$-valuation. Then $2\sin^2(B\theta)=3/2$, so $T=3$ reduces to the same equation; since $C>B=2n/3$, we now have $1\le r<n/6<n/4$. Therefore
\[
  T\in\Q
  \quad\Longrightarrow\quad
  T=1
  \quad\text{or}\quad
  (n,B,C)=(6,2,4),\ T=\frac13.
\]

\end{proof}

\subsection{Classification of \texorpdfstring{$T=1$}{T=1}}

\begin{lemma}\label{lem:T-one}
Under the same conditions,
\[
 T(n,B,(n-C)/2)=1\quad\Longleftrightarrow\quad(n,B,C)=(12,4,6).
\]
\end{lemma}

\begin{proof}
The condition $T=1$ is
\[
  2\sin^2(B\theta)\cos(C\theta)=\sin(4\theta)\sin(C\theta).
  \tag{\ref{thm:Q1Opp}.2}
\]
First, if $B\le2$, then since $C>B$,
\[
  T<\frac{2\sin^2(B\theta)\cot(B\theta)}{\sin(4\theta)}
  =\frac{\sin(2B\theta)}{\sin(4\theta)}\le1,
\]
with equality not occurring. Hence $B>2$. If $n-C=2$, then $B\le C-2=n-4$ and
\[
 T=\frac{\sin^2(B\theta)}{\cos^2(2\theta)}
 \le\frac{\cos^2(4\theta)}{\cos^2(2\theta)}<1.
\]
Thus $n-C>2$ as well.

Put the following six positive rationals:
\[
  u_1=\frac{B-2}{2n},\quad
  u_2=\frac{B}{2n},\quad
  u_3=\frac{n-C-2}{2n},
\]
\[
  v_1=\frac{n-B+2}{2n},\quad
  v_2=\frac1n,
  \quad
  v_3=\frac{C-B}{2n}.
\]
These are positive and sum to $1$. Writing $x=B\theta$, $y=C\theta$ and $z=2\theta$, the difference between the two three-sine products is
\[
\begin{aligned}
 &\sin(x-z)\sin x\cos(y+z)
   -\cos(x-z)\sin z\sin(y-x)\\
 &\qquad=\sin^2x\cos y-\sin z\cos z\sin y.
\end{aligned}
\]
Twice this difference is $2\sin^2(B\theta)\cos(C\theta)-\sin(4\theta)\sin(C\theta)$. Hence, without division by any factor, equation~(\ref{thm:Q1Opp}.2) is equivalent to
\[
  \sin(\pi u_1)\sin(\pi u_2)\sin(\pi u_3)
  =
  \sin(\pi v_1)\sin(\pi v_2)\sin(\pi v_3).
\]
Apply Theorem~\ref{thm:PR}. A trivial solution requires equality of the following multisets, with multiplicities retained:
\[
  \{B-2,B,n-C-2\}=\{n-B+2,2,C-B\},
\]
but this contradicts the conditions. Indeed, since the right-hand side contains $2$, one of the entries of the left-hand side equals $2$. $B=2$ contradicts $B>2$. If $B-2=2$ then $B=4$, and the remaining two-entry multiset $\{4,n-C-2\}$ would have to equal $\{n-2,C-4\}$. Here $C\ge6$, while $n-C>2$ and the parity condition give $n-C\ge4$, so $n\ge10$. Thus $n-2>4$ and $n-2>n-C-2$, ruling out this equality. Finally, if $n-C-2=2$ then $C=n-4$, and the remaining comparison requires $2B=n+4$ and $2B=n-4$, or $2B=n-2$ and $2B=n+2$, simultaneously, both of which are contradictory.

Matching against the four one-parameter families of Poonen--Rubinstein's Table 3, allowing permutations within each triple and interchange of the two triples, the only one that remains is the first family's
\[
  (n,B,C)=(12,4,6).
\]
It matches none of the $65$ sporadic solutions of Table 4. The finite matching against both tables is recorded, with its verification procedure, in Appendix~\ref{app:finite-matching}. Conversely, for $(n,B,C)=(12,4,6)$ we have $t=(n-C)/2=3$, and direct substitution gives $T=1$. Hence
\[
  T=1\quad\Longleftrightarrow\quad (n,B,C)=(12,4,6).
\]

\end{proof}

\begin{proof}[Proof of Theorem~\ref{thm:Q1Opp}]
By Proposition~\ref{prop:a1-formula}, $a_1$ is rational if and only if $T$ is rational. Lemmas~\ref{lem:T-rational-values} and \ref{lem:T-one} give exactly the two configurations.

Returning to $(n,k,m)$: for the exception $(n,B,C)=(6,2,4)$, we have $l=2$, and from $2k+l=4$, $k=1$, $m=(6-2)/2=2$, so $(n,k,m)=(6,1,2)$. Here $T=1/3$, so
\[
  a_1=\frac tn-\frac{T}{n}=\frac16-\frac1{18}=\frac19.
\]
In the case $T=1$, $(n,B,C)=(12,4,6)$ gives $l=4$, $k=(C-B)/2=1$, $m=(12-4)/2=4$, so $(n,k,m)=(12,1,4)$. Here $t=3$ and $T=1$, so
\[
  a_1=\frac{3}{12}-\frac1{12}=\frac16.
\]
\end{proof}

\section{Proof of Theorem A and geometric interpretations}\label{sec:profile-proof}

\begin{proof}[Proof of Theorem~\ref{thm:profile}]
The four blocks of the table follow from Theorems~\ref{thm:Q0}, \ref{thm:Q2}, \ref{thm:Q1Opp}, and \ref{thm:Adj}, respectively. All remaining cases follow from the identities stated in Theorem~\ref{thm:profile} and complementation.
\end{proof}

\begin{corollary}[Simultaneous rationality of all four regions]\label{cor:all-rational}
\[
  a_0,a_1,a_2,a_3\in\Q
\]
hold simultaneously if and only if
\[
  (n,k,m)=(6,1,2).
\]
In this case
\[
  (a_0,a_1,a_2,a_3)=\left(\frac1{18},\frac19,\frac{13}{18},\frac19\right).
\]
\end{corollary}

\begin{proof}
For all four regions to be rational, in particular $(n,k,m)$ must belong to $\mathcal E_0\cap\mathcal E_1\cap\mathcal E_2$. From the configurations in Theorem~\ref{thm:profile} this intersection is only $(6,1,2)$. For this configuration, the values from Theorem~\ref{thm:Q0}, Theorem~\ref{thm:Q2}, and Theorem~\ref{thm:Q1Opp} give the displayed quadruple.
\end{proof}

% Independently checked geometric interpretations.  These remarks can be
% inserted into the corresponding classification sections of main.tex.

\begin{remark}[Diagonals next to a diameter]\label{rem:adj-geometry}
The infinite family in $\mathcal E_{\mathrm{adj}}$ has a direct geometric
explanation.  Suppose that $n\ge6$ is even and $m=n/2-1$.
The diameter $V_0V_{n/2}$ divides $P_n$ into two equal halves, and
$Q_0\cup Q_1$ is the half with boundary $V_0V_1\cdots V_{n/2}$ with
the triangle $V_0V_mV_{n/2}$ removed.  This triangle has base $2$ and
altitude $\sin(2\pi/n)$, so its area is $2\Area(P_n)/n$.  Consequently
$a_0+a_1=1/2-2/n$, independently of $k$.
For the exceptional pair $(n,m)=(12,3)$, the chord $V_0V_2$ has length
equal to the circumradius, while $V_0V_{2m}$ is a diameter.  Hence their
length ratio $U(12,3)$ is $2$, giving $a_0+a_1=1/12$.
\end{remark}

\begin{remark}[A reflected central triangle]\label{rem:Q2-geometry}
Suppose that $n\ge8$ is even, $m=n/2-2$, and $k=n/2-3$;
equivalently, $l=4$ and $t=1$.  Put $E=V_m$ and
$F=V_{n-m+k}=V_{n-1}$, and let $O$ be the centre of $P_n$.
The minor central angle subtended by $EF$ is $\pi-2\pi/n$,
so the base angles of $\triangle EOF$ are $\pi/n$.
The base angles of $\triangle EXF$ are also $t\pi/n=\pi/n$.
Since $O$ and $X$ lie on opposite sides of $EF$, $X$ is the reflection
of $O$ in $EF$.  In particular, $XE=XF=1$ and
\[
  \frac{\Area(EXF)}{\Area(P_n)}
  =\frac{\Area(EOF)}{\Area(P_n)}=\frac1n.
\]
The boundary arc of $Q_2$ has $k+l=n/2+1$ edges.  Its cap, bounded by
this arc and $EF$, has area ratio
\[
  \frac{n/2+1}{n}
  -\frac{\sin(\pi+2\pi/n)}{n\sin(2\pi/n)}
  =\frac12+\frac2n.
\]
Adding the reflected triangle therefore gives $a_2=1/2+3/n$.
Figure~\ref{fig:infinite-families} illustrates both geometric
decompositions.
\end{remark}

\begin{figure}[htb]
  \centering
  \includegraphics{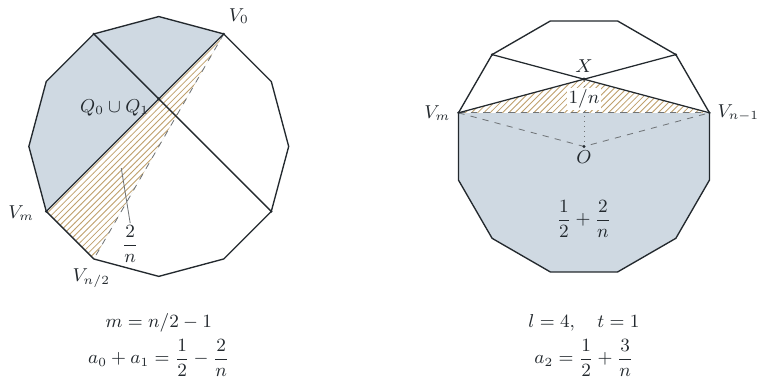}
  \caption{The infinite families illustrated in a regular dodecagon. Left: a half-polygon with a triangle of area ratio $2/n$ removed. Right: a cap of area ratio $1/2+2/n$ together with the reflection of a central triangle of area ratio $1/n$.}
  \label{fig:infinite-families}
\end{figure}

\section{Concluding remarks}\label{sec:conclusion}

The rationality profile shows that different regions in the same polygon dissection can have different arithmetic behavior. The finite classification for $Q_0$ and the two geometric families for $Q_2$ and adjacent sums make this distinction explicit. Congruence of the diagonals supplies the reflection symmetry and keeps the area reduction within short trigonometric relations.

Several related questions remain. One may ask for an analogous classification for two intersecting diagonals of unequal lengths, where the lateral areas need not agree. Dissections by three or more diagonals introduce additional incidence conditions and suggest a closer comparison with the concurrency problem studied by Poonen--Rubinstein. A further question is to classify configurations for which a ratio $a_i/a_j$ between two region areas is rational, even when neither area is a rational fraction of the whole polygon. These questions lie beyond the classification proved here.

As a separate consistency check, the verification program \texttt{verify\_geometry.py} \cite{Watanabe2026Code} evaluates every admissible configuration with $5\le n\le80$ at $60$-digit decimal precision. For all $19\,760$ configurations, areas computed from vertex coordinates, line intersections, and the shoelace formula agree with Proposition~\ref{prop:area-reduction}; the rational values listed in Theorem~\ref{thm:profile} also agree, within a tolerance of $10^{-50}$. These finite numerical checks are not used in the proof of completeness.

\backmatter

\section*{Statements and Declarations}

\bmhead*{Funding}
No funding was received for this work.

\bmhead*{Competing interests}
The authors declare no competing interests.

\bmhead*{Author contributions}
Shoichi Inoue proposed and formulated the research problem. Both authors contributed to the development of the mathematical arguments and proofs with assistance from generative AI tools. Keiichi Watanabe carried out the computational checks and prepared the initial draft, both with AI assistance. Both authors reviewed the manuscript and checked its mathematical correctness. Both authors approved the final version of the manuscript.

\bmhead*{Use of artificial intelligence}
During the preparation and revision of this manuscript, the authors used ChatGPT, Claude, and Gemini to obtain suggestions on manuscript organization and mathematical exposition, and to assist in identifying potential gaps and exploring revisions to individual proof arguments. These tools were also used to explore candidate configurations, suggest proof strategies, and assist in developing code for finite verification. The authors independently checked all mathematical arguments and references incorporated into the manuscript, revised the text as appropriate, and take full responsibility for the final content.

\bmhead*{Data availability}
The finite-matching data, verification programs, instructions, and recorded outputs are publicly archived on Zenodo as \emph{Polygon area profiles: verification code}, version~v1.0.0~\cite{Watanabe2026Code}. The exact checks use rational arithmetic; a separate program checks the area formulas and listed values numerically. Online Resource~1 additionally includes the manuscript source for table comparison.

\bmhead*{Supplementary information}
\textbf{Online Resource 1.} Exact verification of the Myerson and Poonen--Rubinstein finite matching and numerical checks of the polygon area formulas and classified values. The archive \texttt{ESM\_1.zip} contains the public verification code and data from version~v1.0.0~\cite{Watanabe2026Code}, together with the current manuscript source and instructions for the additional table comparison.

\begin{appendices}

\section{Finite matching for \texorpdfstring{$T=1$}{T=1}}\label{app:finite-matching}

We verify the comparison with Tables 3 and 4 of Poonen--Rubinstein~\cite{PoonenRubinstein1998} used in Theorem~\ref{thm:Q1Opp}. The target is
\[
 (u_1,u_2,u_3;v_1,v_2,v_3)=
 \left(\frac{B-2}{2n},\frac{B}{2n},\frac{n-C-2}{2n};
 \frac{n-B+2}{2n},\frac1n,\frac{C-B}{2n}\right).
\]
In particular, every matching must satisfy
\begin{equation}\label{eq:matching-tests}
 u_1+v_1=\frac12,\qquad u_2-u_1=v_2>0.
\end{equation}
Recover $n=1/v_2$, $B=2nu_2$ and $C=B+2nv_3$ for any surviving assignment, and check the entire target and the conditions of Theorem~\ref{thm:Q1Opp}.

\subsection{The four one-parameter families}

Write $s$ for the parameter called $t$ in Poonen--Rubinstein's Table 3, to distinguish it from the geometric variable $t=m-k$. In their order, the two triples are
\[
\begin{array}{c|c|c}
 & \text{First triple} & \text{Second triple}\\
\midrule
\mathrm{I} &(1/6,s,1/3-2s)&(1/3+s,s,1/6-s)\\
\mathrm{II} &(1/6,1/2-3s,s)&(1/6-s,2s,1/6+s)\\
\mathrm{III} &(1/6,1/6-2s,2s)&(1/6-2s,s,1/2+s)\\
\mathrm{IV} &(1/3-4s,s,1/3+s)&(1/6-2s,3s,1/6+s)
\end{array}
\]
The ranges are $0<s<1/6$ for I and II, and $0<s<1/12$ for III and IV. For each family, add each of the three entries of the first triple to each of the three entries of the second. Every constant cross-sum equals $1/6$, so none is identically $1/2$. Equating the nine affine cross-sums to $1/2$ therefore leaves only the finitely many parameter values listed in the second column below. Thus permutations and interchange of the triples are both covered.
\begin{center}
\begin{tabular}{@{}ccl@{}}
\toprule
Family & Parameters passing $u_1+v_1=1/2$ & Final matching\\
\midrule
I & $1/12$ & $(n,B,C,s)=(12,4,6,1/12)$\\
II & $1/24,\ 1/12$ & None\\
III & None & None\\
IV & $1/24$ & None\\
\bottomrule
\end{tabular}
\end{center}
For I at $s=1/12$, the only possible ordered pair is $(u_1,v_1)=(1/12,5/12)$: the reverse order leaves no larger entry for $u_2$. The second condition in \eqref{eq:matching-tests} then forces
\[
 (u_1,u_2,u_3;v_1,v_2,v_3)
 =(1/12,1/6,1/6;5/12,1/12,1/12).
\]
This is exactly the target for $(n,B,C)=(12,4,6)$, which satisfies all the required conditions.

For II at $s=1/24$, and IV at $s=1/24$, the triples, as multisets, are
\[
 \{1/24,1/6,3/8\},\qquad \{1/12,1/8,5/24\}.
\]
The only cross-pair summing to $1/2$ is $3/8+1/8$. Taking $u_1=3/8$ leaves no larger entry for $u_2$. Taking $u_1=1/8$ forces $u_2=5/24$ and hence $v_2=1/12$, which is absent from the other triple. For II at $s=1/12$, both triples are $\{1/12,1/6,1/4\}$, giving a trivial solution as noted by Poonen--Rubinstein; the only cross-pair is $1/4+1/4$, again leaving no $u_2>u_1$. This proves the stated family classification.

\subsection{The 65 sporadic solutions}

We use Table 4 on printed page 14 of the author-hosted version of Poonen--Rubinstein's paper~\cite{PoonenRubinstein1998}, dated November 18, 1997 and available at \url{https://math.mit.edu/~poonen/papers/ngon.pdf}. Number its rows from top to bottom as $j=1,\ldots,65$, without restarting at a change of denominator. The data file \texttt{pr-table4.csv} \cite{Watanabe2026Code} records these numerator triples and denominators, and \texttt{verify\_exact.py} recomputes every $d_j$ below as well as the four-family matching. For row $j$, let $D_j$ be its listed common denominator, with numerator triples $(a_1,a_2,a_3)$ and $(b_1,b_2,b_3)$. Define
\[
 d_j=\min_{1\le r,q\le3}\left|a_r+b_q-\frac{D_j}{2}\right|.
\]
All the listed denominators are even. A match requires $d_j=0$ by \eqref{eq:matching-tests}. The following table records $d_j$ for every row; within each range the values are in increasing row order.
\begin{center}
\begin{tabular}{@{}ccl@{}}
\toprule
Rows $j$ & $D_j$ & Values of $d_j$\\
\midrule
1--13 & 30 & $1,1,1,1,6,1,1,1,7,1,1,1,2$\\
14--19 & 42 & $1,1,1,1,2,6$\\
20--30 & 60 & $3,1,1,3,2,2,2,2,2,3,2$\\
31--41 & 60 & $2,9,5,1,1,1,1,5,1,1,3$\\
42--49 & 84 & $7,2,2,11,1,1,1,5$\\
50--53 & 90 & $1,3,3,7$\\
54--61 & 120 & $6,4,4,19,1,1,1,7$\\
62--65 & 210 & $3,2,7,27$\\
\bottomrule
\end{tabular}
\end{center}
Each entry requires only nine integer sums. For example, row 1 has numerator triples $(3,4,9)$ and $(4,5,5)$ with $D_1=30$. Their cross-sums form the set $\{7,8,9,13,14\}$, whose minimum distance from $15$ is $1$. Thus $d_1=1$. The same test gives the displayed positive value in each of the remaining 64 rows.

Each $d_j$ is positive and invariant under permutations within the triples or interchange of the triples. Hence none of the 65 sporadic solutions matches the target.

\end{appendices}

\begingroup
\interlinepenalty=10000 % Keep each reference entry on a single page.
\begin{CJK*}{UTF8}{ipxm}
\bibliography{references}
\end{CJK*}
\endgroup

\end{document}